\documentclass[12pt]{article}

\usepackage{amsmath}
\usepackage{amssymb}
\usepackage{microtype}
\usepackage{amsthm}
\usepackage[margin=1.5in]{geometry}
\usepackage[usenames]{color}
\usepackage{graphicx}
\usepackage{tikz}

\newtheorem{thm}{Theorem}[section]

\newtheorem{lemma}[thm]{Lemma}

\newtheorem{conj}[thm]{Conjecture}

\newtheorem{question}[thm]{Question}

\newcommand{\rst}[1]{\ensuremath{{\mathbin\upharpoonright}%
\raise-.5ex\hbox{$#1$}}}

\begin{document}

\renewcommand{\thefootnote}{\fnsymbol{footnote}}

\title{Asymptotic normality of some graph sequences}

\author{David Galvin\thanks{dgalvin1@nd.edu; Department of Mathematics,
University of Notre Dame, Notre Dame IN 46556. Research supported by NSA grant H98230-13-1-0248, and by the Simons Foundation.}}

\date{\today}

\maketitle

\begin{abstract}
For a simple finite graph $G$ denote by ${G \brace k}$ the number of ways of partitioning the vertex set of $G$ into $k$ non-empty independent sets (that is, into classes that span no edges of $G$). If $E_n$ is the graph on $n$ vertices with no edges then ${E_n \brace k}$ coincides with ${n \brace k}$, the ordinary Stirling number of the second kind, and so we refer to ${G \brace k}$ as a {\em graph Stirling number}.

Harper showed that the sequence of Stirling numbers of the second kind, and thus the graph Stirling sequence of $E_n$, is {\em asymptotically normal} --- essentially, as $n$ grows, the histogram of $\left({E_n \brace k}\right)_{k \geq 0}$, suitably normalized, approaches the density function of the standard normal distribution.

In light of Harper's result, it is natural to ask for which sequences $(G_n)_{n \geq 0}$ of graphs is there asymptotic normality of $\left({G_n \brace k}\right)_{k \geq 0}$. Do and Galvin conjectured that if for each $n$, $G_n$ is acylic and has $n$ vertices, then asymptotic normality occurs, and they gave a proof under the added condition that $G_n$ has no more than $o(\sqrt{n/\log n})$ components.

Here we settle Do and Galvin's conjecture in the affirmative, and significantly extend it, replacing ``acyclic'' in their conjecture with ``co-chromatic with a quasi-threshold graph, and with negligible chromatic number''. Our proof combines old work of Navon and recent work of Engbers, Galvin and Hilyard on the normal order problem in a Weyl algebra, and work of Kahn on the matching polynomial of a graph.
\end{abstract}

\section{Introduction}

Let $G=(V,E)$ be a (simple, finite, loopless) graph. An {\em independent set} in $G$ is a subset of the vertices, no two of which are adjacent. For each integer $k$ set
$$
{G \brace k} = \left|\{\mbox{partitions of $V$ into $k$ non-empty independent sets}\}\right|.
$$
Equivalently, ${G \brace k}$ is the number of proper $k$-colorings of $G$ that use all $k$ colors, with two colorings identified if they are identical up to the names of the colors. As far as we know, this parameter was first explicitly considered by Tomescu \cite{Tomescu}. When $G=E_n$, the $n$-vertex graph with no edges, ${G \brace k}$ is just the Stirling number of the second kind ${n \brace k}$, the number of partitions of a set of size $n$ into $k$ non-empty classes; for this reason we refer to ${G \brace k}$ as a {\em graph Stirling number}, and to the sequence $\left({G \brace k}\right)_{k \in {\mathbb Z}}$ as the {\em Stirling sequence (of the second kind)} of $G$. For a brief history of the study of the Stirling sequence of graphs, see \cite{GalvinThanh} and the references therein.

A seminal result in the study of the (ordinary) Stirling numbers of the second kind is Harper's theorem \cite{Harper}, which concerns asymptotic normality. Suppose that for each $n \geq 0$ we have a sequence $s_n = (a_{n,k})_{k \in {\mathbb Z}}$ of non-negative terms with $0 < \sum_{k \in {\mathbb Z}} a_{n,k} < \infty$. Informally, asymptotic normality of $s_n$ means that its histogram, suitably normalized, approaches the density function of the standard normal distribution as $n$ grows. Formally, associate with each $n$ a random variable $X_n$ taking values on ${\mathbb Z}$ by
$$
\Pr\left(X_n = k\right) = \frac{a_{n,k}}{\sum_{j \geq 0} a_{n,j}}.
$$
(Note that if a set consists of objects of sizes $0, 1, 2, \ldots$, with the objects of size $n$ divided into classes (indexed by $0, 1, 2, \ldots$), and if $a_{n,k}$ counts the number of objects of size $n$ that are in the $k$th class, then $X_n$ may be interpreted as observing the class of a uniformly chosen element of size $n$.) We say that $s_n$ is {\em asymptotically normal} if $X_n$ approaches a normal distribution in probability as $n$ grows, that is, if for all $x \in {\mathbb R}$ we have uniformly in $x$
$$
\Pr\left(\frac{X_n-\mu_n}{\sigma_n} \leq x\right) \rightarrow \Pr(Z \leq x)
$$
as $n \rightarrow \infty$, where $\mu_n$ and $\sigma_n$ are the mean and standard deviation of $X_n$, and $Z$ is the standard normal random variable. From \cite{Harper} we have the following.
\begin{thm} \label{thm-Harper}
The ordinary Stirling sequence of the second kind, $\left({n \brace k}\right)_{k \in {\mathbb Z}}$, is asymptotically normal.
\end{thm}
Equivalently, the Stirling sequence of the empty graph, $\left({E_n \brace k}\right)_{k \in {\mathbb Z}}$, is asymptotically normal, raising a natural question.
\begin{question} \label{quest-an}
For which sequences $(G_n)_{n \geq 0}$ of graphs is the Stirling sequence $\left({G_n \brace k}\right)_{k \in {\mathbb Z}}$ asymptotically normal?
\end{question}

We might expect that if $G_n$ is obtained from $E_n$ by a suitably small perturbation, then asymptotic normality should be preserved. Thinking along these lines, Do and Galvin \cite{GalvinThanh} conjectured the following  extension of Theorem \ref{thm-Harper}.
\begin{conj} \label{conj-GalvinThanh}
If $G_n$ is an acyclic graph on $n$ vertices then $\left({G_n \brace k}\right)_{k \in {\mathbb Z}}$ is asymptotically normal.
\end{conj}

A standard approach to showing asymptotic normality of a sequence $s_n = (a_{n,k})_{k \in {\mathbb Z}}$ that is supported on a finite subset of ${\mathbb N}$ is to show that the generating polynomial $\sum_{k \in {\mathbb Z}} a_{n,k} x^k$ has all real zeros, and so factors into linear terms over the reals. This implies that $X_n$ can be represented as a sum of independent Bernoulli random variables, and the central limit theorem then shows that asymptotic normality is implied by the condition $\sigma_n \rightarrow \infty$ as $n \rightarrow \infty$. (This approach, which has often been rediscovered, is originally due to Levy \cite{Levy}.)

For acyclic $G_n$, showing that $\sum_{k \in {\mathbb Z}} {G_n \brace k} x^k$ has only real zeros is not too hard,
but due to the complexity of the expressions involved Do and Galvin were only able to establish the condition $\sigma_n \rightarrow \infty$ under the addition assumption that $G_n$ has no more than $c\sqrt{n/\log n}$ components for some suitably small $c$, and thus were only able to establish Conjecture \ref{conj-GalvinThanh} under this assumption.

Here we take a different approach that allows us to prove Conjecture \ref{conj-GalvinThanh}, and to generalize it considerably. To state our main theorem, we need a little notation. We begin by introducing the family of {\em quasi-threshold} graphs, which we define inductively by three rules:
\begin{enumerate}
\item The graph $K_1$ is a quasi-threshold graph.
\item If $G$ is quasi-threshold, and $G'=G+K_1$, the graph obtained from $G$ by adding a dominating vertex (a new vertex adjacent to all the vertices of $G$), then $G'$ is quasi-threshold.
\item If $G_1$ and $G_2$ are quasi-threshold, then $G_1 \cup G_2$, the disjoint union of $G_1$ and $G_2$, is quasi-threshold.
\end{enumerate}
Quasi-threshold graphs, which are sometimes called {\em trivially perfect} graphs, are well-known and well-studied; see e.g. \cite{Jing-HoJer-JeongChang} for more information. Note that if a graph is constructed using only rules 1 and 2 above, then it is an example of a {\em threshold} graph.

Write $\chi_G(x)$ for the chromatic polynomial of $G$ and $\chi(G)$ for its chromatic number. Say that $G$ and $H$ are {\em co-chromatic} if $\chi_G(x) \equiv \chi_H(x)$. Our main theorem is the following.

\begin{thm} \label{thm-main}
For each $n$, let $G_n$ be either a quasi-threshold graph, or co-chromatic with some quasi-threshold graph. Let $f(n)$ be the number of vertices of $G_n$, and let $g(n)=\chi(G_n)/f(n)$. If $f(n) \rightarrow \infty$ and $g(n) \rightarrow 0$ as $n \rightarrow \infty$ then the Stirling sequence of $G_n$ is asymptotically normal.
\end{thm}
To see that this implies Conjecture \ref{conj-GalvinThanh}, let $G_n$ be an acyclic graph on $n$ vertices. We clearly have $f(n) \rightarrow \infty$ as $n \rightarrow \infty$, and since $\chi(G_n) \leq 2$ we also have $g(n) \rightarrow 0$. If $G_n$ has $k$ components then $\chi_{G_n}(x) = x^k(x-1)^{n-k}$ and so is co-chromatic with the quasi-threshold graph $K_{1,n-k} \cup K_1 \cup \ldots \cup K_1$ (a star on $n-k+1$ vertices together with $k-1$ isolated vertices).

The proof of Theorem \ref{thm-main} unexpectedly passes through the normal order problem in a Weyl algebra, and in particular it makes use of four different combinatorial interpretations, due to Navon and to Engbers, Galvin and Hilyard, of the normal order coefficients of an arbitrary word in the algebra. These interpretations allow us to convert the graph Stirling number of a quasi-threshold graph into a count of matchings in a certain bipartite graph. A very general result of Kahn on asymptotic normality of matching sequences then completes the proof. In Section \ref{sec-tools} we review all the necessary background to understand these results, and the (short) proof of Theorem \ref{thm-main} is given in Section \ref{sec-proof}.

\section{The pieces of the puzzle} \label{sec-tools}

The proof of Theorem \ref{thm-main} involves various aspects of each of four different combinatorial interpretations the normal order of a word in a Weyl algebra. We begin by explaining these terms.

A {\em Weyl algebra} is generated by two symbols, which we shall call $x$ and $D$, satisfying the single relation $Dx=xD+1$. The choice of symbol names is motivated by the fact that we may represent the Weyl algebra as a set of operators on a space of infinitely differentiable functions in a single variable $x$ by interpreting the symbol ``$x$'' as multiplication by $x$ and ``$D$'' as differentiation with respect to $x$ (and ``$1$'' as the identity); so, for example, the word $Dx$, when applied to a function $f(x)$, results in $(d/dx)(xf(x)) = xf'(x) + f(x)$. This is the same result as would be obtained by applying $xD + 1$, justifying that in this representation we have the relation $Dx=xD+1$.

If $w$ is a word in a Weyl algebra with $m$ $x$'s and $n$ $D$'s, then one can show by induction on the length of $w$ that it has a unique representation of the form
\begin{equation} \label{eq-normal-order}
w = x^{m-n} \sum_{k \in {\mathbb Z}} S_w(k) x^kD^k,
\end{equation}
called the {\em normal order} of $w$. The study of the normal order of words goes back to the 1800's, and has recently seen significant activity owing to its occurrence in quantum mechanics; see e.g. \cite{BlasiakHorzelaPensonSolomonDuchamp} for an introduction to this perspective.

Numerous combinatorial interpretations for the coefficients $S_w(k)$ from (\ref{eq-normal-order}) have been given. Here we explain the four that are of interest to us for the proof of Theorem \ref{thm-main}. We will confine our discussion to those $w$ which are {\em Dyck words} --- words with the same number of $x$'s as $D$'s, and such that, reading the word from left to right, every initial segment has at least as many $x$'s as $D$'s. We will use $n$ for the number of $x$'s in $w$.

All but the first combinatorial interpretation depend on a certain representation in ${\mathbb Z}^2$ of a Dyck word. A {\em Dyck path} in ${\mathbb R}^2$ is a staircase path (a path that proceeds by taking unit steps, either in the positive $x$ direction or the positive $y$ direction) that starts at $(0,0)$, ends on the line $x=y$, any never goes below this line. There is a natural correspondence between Dyck paths and Dyck words, given by mapping steps in the positive $y$ direction to $x$, and steps in the positive $x$ direction to $D$. For example, the word $xxDxxDxDDD$ (which we will use as a running example for our interpretations) corresponds to the path that goes from $(0,0)$ to $(0,1)$ to $(0,2)$ to $(1,2)$ to $(1,3)$ to $(1,4)$ to $(2,4)$ to $(2,5)$ to $(3,5)$ to $(4,5)$ to $(5,5)$.

\subsection{$S_w(k)$ in terms of partitions of a quasi-threshold graph} \label{subsec-word}

To a Dyck word $w$ we can naturally associate a quasi-threshold graph $G_w$ inductively as follows.
\begin{enumerate}
\item If $w=xD$, then $G_w = K_1$.
\item If $w$ can be written in the form $xw'D$ with $w'$ a non-empty Dyck work, then $G_w = G_{w'} + K_1$.
\item If $w$ can be written in the form $w_1\ldots w_\ell$ with each $w_i$ a shorter non-empty Dyck word, then $G_w = G_{w_1} \cup \ldots \cup G_{w_\ell}$.
\end{enumerate}
For example, if $w=xxDxxDxDDD$ then we construct $G_w$ by adding a dominating vertex to the graph associated with $xDxxDxDD$. This is the union of $K_1$ (the graph associated with $xD$) and the graph associated with $xxDxDD$. This latter is obtained by adding a dominating vertex to $K_1 \cup K_1$ (the graph associated with $xDxD$), so is a path on three vertices. The graph we end up with has a vertex, $v_1$ say, adjacent to each of four vertices, $v_2$, $v_3$, $v_4$, $v_5$ say, with two edges among these four vertices, $v_3v_4$, $v_4v_5$ say, inducing a path on three vertices; call this graph $G^{\rm ex}$.

From \cite{EngbersGalvinHilyard} we have the following.
\begin{thm} \label{thm-fromEGH}
For every Dyck word $w$ and integer $k$, $S_k(w) = {G_w \brace k}$.
\end{thm}
For example, $S_3(xxDxxDxDDD) = {G^{\rm ex} \brace 3} = 2$ (the two partitions of $V(G^{\rm ex})$ into three non-empty independent sets being $v_1|v_2v_4|v_3v_5$ and $v_1|v_2v_3v_5|v_4$).

\medskip

As well as going from Dyck words to quasi-threshold graphs, we will need to go in the other direction.
\begin{lemma} \label{lem-reverse}
If $G$ is a quasi-threshold graph on $n$ vertices then there is a Dyck word $w(G)$ with $n$ $x$'s and $n$ $D$'s such that $G_{w(G)}=G$.
\end{lemma}

\begin{proof}
We proceed by induction on the number of vertices of $G$. If $G=K_1$, we just take $w(G)=xD$. If $G$ on more than one vertex is of the form $G' + K_1$ for some quasi-threshold graph $G'$ then we take $w$ to be $xw'D$ where $w'=w(G')$. If $G$ on more than one vertex breaks into components $G_1, \ldots, G_\ell$, each a quasi-threshold graph, then we take $w$ to be $w_1w_2\ldots w_\ell$ where for each $i$ $w_i=w(G_i)$.
\end{proof}
For example, because $G^{\rm ex}$ has a dominating vertex ($v_1$) we have $w(G^{\rm ex}) = xw'D$ where $w'=w(G^{\rm ex}-v_1)$. Since $G^{\rm ex}-v_1$ has components $K_1$ and $P_3$ (the path on three vertices), $w'$ is the concatenation of $xD$ and $w''=w(P_3)$. Since $P_3$ has a dominating vertex joined to two isolated vertices, $w'' = x\left((xD)(xD)\right)D$. Putting all this together leads to $w(G)=xxDxxDxDDD$, as we would expect.

\subsection{$S_w(k)$ in terms of rook placements on a Ferrers board}

Label each unit square in ${\mathbb Z}^2$ with the coordinates of its top-right corner (so, for example, the square with corners at $(0,0)$, $(1,0)$, $(0,1)$ and $(1,1)$ gets label $(1,1)$). Let $B_w$ be the set of labels of the unit squares that lie above the Dyck (staircase) path of $w$ and inside the box $[0,n] \times [0,n]$ (note that $B_w$ forms a Ferrers board), and let $r_k(B_w)$ be the number of ways of placing $k$ non-attacking rooks on $B_w$ (that is, the number of ways of selecting a subset of $B_w$ of size $k$, with no two elements of the subset sharing a first coordinate, and no two sharing a second coordinate). Navon \cite{Navon} proved the following.
\begin{thm} \label{thm-fromNavon}
For every Dyck word $w$ and integer $k$, $S_k(w) = r_{n-k}(B_w)$.
\end{thm}
For example, if $w=xxDxxDxDDD$ then $B_w=\{(1,3),(1,4),(1,5),(2,5)\}$, and $S_3(xxDxxDxDDD) = r_2(\{(1,3),(1,4),(1,5),(2,5)\}) = 2$ (the two valid rook placements of size two being $\{(1,3),(2,5)\}$ and $\{(1,4),(2,5)\}$).

\subsection{$S_w(k)$ in terms of partitions of a clique-union graph} \label{subsec-H}

Let $W_w$ be the set of (labels of) unit squares that lie below the Dyck path of $w$, and completely above the line $x=y$. Define a graph $H_w$ on vertex set $\{1, \ldots, n\}$ by putting an edge from $i$ to $j$ ($i < j$) if and only if $(i,j) \in W_w$. For example, if $w=xxDxxDxDDD$ then $W_w=\{(1,2),(2,3),(2,4),(3,4),(3,5),(4,5)\}$ and $H_w$ is the graph on vertex set $\{1,2,3,4,5\}$ with edge set $\{12,23,24,34,35,45\}$; call this graph $G^{\rm ex'}$ (notice that $G^{\rm ex'}$ is not isomorphic to $G^{\rm ex}$, since the former does not have a dominating vertex but the latter does). From \cite{EngbersGalvinHilyard} we have the following.
\begin{thm} \label{thm-fromEGH2}
For every Dyck word $w$ and integer $k$, $S_k(w) = {H_w \brace k}$.
\end{thm}
So, for example, $S_3(xxDxxDxDDD) = {G^{\rm ex'} \brace 3} = 2$ (the two partitions of $V(G^{\rm ex'})$ into three non-empty independent sets being $13|25|4$ and $14|25|3$).

\medskip

It is worth noting that $H_w$ is determined by the places where the Dyck path of $w$ takes a step up followed by a step to the right. To make this precise, say that the Dyck path of $w$ {\em turns around} the unit square labeled $(x,y)$ if it takes a step from $(x-1,y-1)$ to $(x-1,y)$ and then steps to $(x,y)$. Let $T_w=\{(x_1,y_1), \ldots, (x_k,y_k)\}$ be the set of (labels of) unit squares that the path of $w$ turns around. Then it is easy to see that the edge set of $H_w$ can be covered by putting a clique on each of the consecutive segments $\{x_i, \ldots, y_i\}$, $1 \leq i \leq k$. (It is for this reason that we refer to $H_w$ as a {\em clique-union} graph). For example, if $w=xxDxxDxDDD$ then $T_w=\{(1,2),(2,4),(3,5)\}$ and $G^{\rm ex'}$ can be constructed by forming cliques on $\{1,2\}$, $\{2,3,4\}$ and $\{3,4,5\}$.

\subsection{$S_w(k)$ in terms of matchings of a bipartite graph} \label{subsec-match}

To $B_w$ associate a bipartite graph $\Gamma_w$, with partition classes $X=\{x_1, \ldots, x_n\}$ and $Y=\{y_1, \ldots, y_n\}$, by putting an edge from $x_i$ to $y_j$ if and only if $(i,j) \in B_w$. A placement of $k$ non-attacking rooks on $B_w$ is easily see to correspond bijectively to a selection of $k$ independent edges (edges sharing no endvertices) in $\Gamma_w$, that is, to a matching of size $k$ in $\Gamma_w$. Write $m_k(\Gamma_w)$ for the number of matchings of size $k$ in $\Gamma_w$. From Theorem \ref{thm-fromNavon} we immediately get the following.
\begin{thm} \label{thm-fromNavon2}
For every Dyck word $w$ and integer $k$, $S_k(w) = m_{n-k}(\Gamma_w)$.
\end{thm}
For example, if $w=xxDxxDxDDD$ then $X=\{x_1, x_2, x_3,x_5\}$, $Y=\{y_1, y_2,y_3,y_4,y_5\}$ and $\Gamma_w$ has edges from $x_1$ to each of $y_3$, $y_4$ and $y_5$ and also an edge from $x_2$ to $y_5$ (so $x_3$, $x_4$, $x_5$, $y_1$ and $y_2$ are all isolated). In this case we get $S_3(xxDxxDxDDD) = m_2(\Gamma_w) = 2$ (the two matchings of size two being $\{x_1y_3,x_2y_5\}$ and $\{x_1y_4,x_2y_5\}$).

\medskip

Taken together these combinatorial interpretations allow us to transform the study of the Stirling sequence of a quasi-threshold graph into the study of the matching polynomial of a bipartite graph, a realm with powerful results on which we can draw. A celebrated result of Heilmann and Lieb \cite{HeilmannLieb} says that for any graph $G$, the polynomial $\sum_{k \in {\mathbb Z}} m_k(G)x^k$ has all real zeros, reducing the problem of showing asymptotic normality to that of showing that the variance of the size a uniformly chosen matching is sufficiently large. This is not an easy task in general; but Kahn \cite{Kahn} found a collection of conditions, in general easier to verify than $\sigma_n \rightarrow \infty$, that imply asymptotic normality of the matching sequence. In particular, from \cite{Kahn} we have the following.
\begin{thm} \label{thm-Kahn}
Let $(G_n)_{n \geq 0}$ be a sequence of graphs all with minimum degree at least one, with $G_n$ having order $v_n$, and matching number (size of largest matching) $\nu_n$. If $v_n \rightarrow \infty$ and $\nu_n \sim v_n/2$ as $n \rightarrow \infty$, then the matching sequence $\left(m_k(G_n)\right)_{k \in {\mathbb Z}}$ is asymptotically normal.
\end{thm}

\section{Putting the pieces together (proof of Theorem \ref{thm-main})} \label{sec-proof}

Let $G_n$ be as given in the statement of Theorem \ref{thm-main}. Without loss of generality we may assume that $G_n$ is quasi-threshold. This is because the chromatic polynomial of a graph $G$ determines its Stirling sequence, and vice-versa; on the one hand, by inclusion-exclusion,
$$
{G \brace k} = \frac{1}{k!}\sum_{i=0}^k (-1)^i{k \choose i}\chi_G(k-i),
$$
while on the other hand, for each positive integer $q$,
$$
\chi_G(q) = \sum_{k \geq 0} {G \brace k}q_{(k)}
$$
where $q_{(k)} = q(q-1) \ldots (q-k+1)$.

Let $w_n$ be the Dyck word with $f(n)$ $x$'s and $f(n)$ $D$'s, given by Lemma \ref{lem-reverse}, satisfying $G_{w_n} = G_n$. Let $H_n$ be the clique-union graph associated with $w_n$, as described in Section \ref{subsec-H}, and let $\Gamma_n$ be the bipartite graph associated with $w_n$, as described in Section \ref{subsec-match}.

Combining Theorems \ref{thm-fromEGH}, \ref{thm-fromEGH2} and \ref{thm-fromNavon2} we have that for each $k$,
\begin{equation} \label{eq-main}
S_{w_n}(k) = {G_n \brace k} = {H_n \brace k} = m_{f(n)-k}(\Gamma_n).
\end{equation}
Using the symmetry of the standard normal, and the fact that the random variable $(X_n-\mu_n)/\sigma_n$ in the definition of asymptotic normality is invariant under a shift in the sequence $(a_{n,k})_{k \in {\mathbb Z}}$, (\ref{eq-main}) shows that asymptotic normality of $\left({G_n \brace k}\right)_{k \in {\mathbb Z}}$ is implied by asymptotic normality of $\left(m_k(\Gamma_n)\right)_{k \in {\mathbb Z}}$.

We cannot (yet) apply Theorem \ref{thm-Kahn}, because $\Gamma_n$ may have isolated vertices. Indeed, if $w_n$ begins with $\ell$ $x$'s in a row and ends with $m$ $D$'s in a row, then from the construction of $\Gamma_n$ it is clear that the isolated vertices of $\Gamma_n$ are exactly $y_1, \ldots, y_\ell$ and $x_{f(n)-m+1}, \ldots, x_{f(n)}$ (as we saw with the example $w=xxDxxDxDDD$ in Section \ref{subsec-match}). Removing these $\ell+m$ vertices we get a graph $\Gamma'_n$, with no isolated vertices, that has the same matching sequence as $\Gamma_n$; we will use Theorem \ref{thm-Kahn} to show asymptotic normality of $\left(m_k(\Gamma'_n)\right)_{k \in {\mathbb Z}}$, which will complete the proof of Theorem \ref{thm-main}.

To apply Theorem \ref{thm-Kahn} we must verify that the order $v_n$ and matching number $\nu_n$ of $\Gamma'_n$ are sufficiently large. We deal first with $v_n$, which is evidently $2f(n)-\ell-m$. Since $f(n) \rightarrow \infty$ by hypothesis, to show $v_n \rightarrow \infty$ we need only show that $\ell$ and $m$ are negligible compared to $f(n)$. From Section \ref{subsec-H} we known that $H_n$ can be constructed by forming various cliques on $\{1, \ldots, f(n)\}$, including one on the $\ell$ vertices $\{1, \ldots, \ell\}$ and one on the $m$ vertices $\{f(n)-m+1, \ldots, f(n)\}$. This shows that $\ell$ and $m$ are both bounded above by the clique number of $H_n$, which is in turn bounded above by $\chi(H_n)$, which (by (\ref{eq-main}), which shows that $G_n$ and $H_n$ are co-chromatic) is equal to $\chi(G_n)$. The hypothesis $g(n) \rightarrow 0$ now gives $v_n \rightarrow \infty$.

We now deal with $\nu_n$. The smallest value of $k$ for which ${G_n \brace k}$, and so by (\ref{eq-main}) $S_{w_n}(k)$, is strictly positive is $k=\chi(G_n)$, which means, again by (\ref{eq-main}), that $\nu_n=f(n)-\chi(G)$. From the last paragraph we know that $2f(n)-2\chi(G_n) \leq v_n \leq  2f(n)$. The hypotheses on $f(n)$ and $g(n)$ now easily give $\nu_n \sim v_n/2$.


\end{document}